\documentclass{amsart}

\usepackage{amsmath,amssymb,amsthm,verbatim}
\usepackage{hyperref}
\usepackage{xcolor}

\newtheorem{theorem}[subsection]{Theorem}
\newtheorem{lemma}[subsection]{Lemma}
\newtheorem{cor}[subsection]{Corollary}
\newtheorem{prop}[subsection]{Proposition}

\theoremstyle{definition}
\newtheorem{definition}[subsection]{Definition}
\newtheorem{remark}[subsection]{Remark}

\newcommand{\haus}{\mathcal{H}}

\newcommand{\spt}{\mathrm{spt}}

\newcommand{\reg}{\mathrm{reg}}
\newcommand{\sign}{\mathrm{sign}}

\newcommand{\graph}{\mathrm{graph}}

\newcommand{\eps}{\epsilon}

\newcommand{\sing}{\mathrm{sing}}

\newcommand{\R}{\mathbb{R}}
\newcommand{\bC}{\mathbf{C}}

\newcommand{\del}{\partial}

\newcommand{\mdiv}{\mathrm{div}}

\newcommand{\cM}{\mathcal{M}}
\newcommand{\N}{\mathbb{N}}
\newcommand{\Z}{\mathbb{Z}}

\title{A Liouville-type theorem for cylindrical cones}
\author{Nick Edelen}
\address{Department of Mathematics, University of Notre Dame, Notre
  Dame, IN 46556}
\email{nedelen@nd.edu}

\author{G\'abor Sz\'ekelyhidi}
\address{Department of Mathematics, Northwestern University, Evanston,
IL 60208}
\email{gaborsz@northwestern.edu}

\begin{document}

\begin{abstract}
  Suppose that $\bC_0^n \subset \R^{n+1}$ is a smooth strictly
  minimizing and strictly stable minimal hypercone, $l \geq 0$, and
  $M$ a complete embedded minimal hypersurface of $\R^{n+1+l}$ lying to one side
  of $\bC = \bC_0 \times \R^l$. If the density at infinity of $M$ is less than twice the
  density of $\bC$, then we show that $M = H(\lambda) \times \R^l$,
  where $\{H(\lambda)\}_\lambda$ is the Hardt-Simon foliation of $\bC_0$. This extends a
  result of L. Simon, where an additional smallness assumption is
  required for the normal vector of $M$. 
\end{abstract}
\maketitle

\section{Introduction}
Liouville type theorems, that is the rigidity properties of entire solutions of
certain partial differential equations, are ubiquitous in geometric
analysis. In this paper we prove a Liouville type theorem for minimal
hypersurfaces lying on one side of a minimal cylindrical hypercone,
extending a recent result of L. Simon~\cite{SimonLiouville}. 

To state the main result, let $\bC_0^n \subset \R^{n+1}$ be a smooth
strictly minimizing and strictly stable minimal hypercone, and let
$\bC = \bC_0 \times \R^l$ for some $l \geq 0$.  Write $\{ H(\lambda)
\}_\lambda$ for the Hardt-Simon foliation~\cite{HS85}
associated to $\bC_0$.  See Section~\ref{sec:background}
 for more details on the notation.

Our main result is the following.   
\begin{theorem}\label{thm:main}
If $M$ is a smooth complete embedded minimal hypersurface of $\R^{n+l+1}$ lying to one
side of $\bC$ satisfying the density bound $\theta_M(\infty) < 2
\theta_\bC(0)$, then $M = H(\lambda) \times \R^l$ for some $\lambda$.  
\end{theorem}

Previously Simon~\cite{SimonLiouville} showed that the same conclusion holds under the additional assumption that the component $\nu_y$ of the normal
vector to $M$ in the $\R^l$ direction is sufficiently small.  The $l=0$ case of the Theorem is due to Hardt-Simon~\cite{HS85}, who proved it for smooth $\bC$ which are merely minimizing.  The \emph{existence} of a foliation associated to a minimizing hypercone $\bC$ was first proven by \cite{BDG} (for quadratic $\bC$), \cite{HS85} (for smooth $\bC$), and just recently \cite{Wang} (for any $\bC$).

The Hardt-Simon foliation and Liouville theorems of \cite{HS85, SimonLiouville} have been of fundamental importance in the analysis of minimal hypersurfaces, including in results concerning generic regularity of stable or minimizing $7$-dimensional hypersurfaces \cite{HS85, Smale, ChLiSp, LiWang}, the construction of stable or minimizing singular minimal hypersurfaces \cite{HS85, SimonSing, Sz21}, and local regularity/tangent cone uniqueness \cite{SimonCAG, EdSp, Sz20, Ed}.

Let us also make the following remark.
\begin{remark}
With only cosmetic changes, Theorem \ref{thm:main} (and all the other lemmas/theorems in this paper) continue to hold for stationary integral varifolds in place of smooth, complete minimal surfaces.  So, if $V$ is a non-zero stationary integral $(n+l)$-varifold in $\R^{n+l+1}$ with $\theta_V(\infty) < 2\theta_\bC(0)$ and $\spt V$ lying to one side of $\bC$, then $V = [H(\lambda) \times \R^l]$ for some $\lambda$.
\end{remark}

Some of the basic ideas and strategies that we use originate from
\cite{Sz20, Sz21}, but the explicit nature of our comparison surfaces
$T_\lambda := H(\lambda) \times \R^l$ allows for significant simplifications. A key
technical tool is a geometric 3-annulus lemma (Lemma~\ref{lem:3ann})
for an excess $E(M,
T_\lambda, R)$ defined for $M$ with respect to $T_\lambda$ at scale
$R$. This in turn depends on a non-concentration estimate
(Theorem~\ref{thm:nc}) to reduce the estimate to the corresponding result for
Jacobi fields.

Given the 3-annulus lemma, the argument can be summarized as follows:
the fact that $M$ lies on one side of $\bC$ implies that the excess of
$M$ with respect to $T_0 = \bC$ grows at most at rate $R^{\gamma-1 +
  \epsilon}$ as $R\to\infty$ for any $\epsilon > 0$. Here $r^\gamma$
is the growth rate of the only positive admissible Jacobi field on
$\bC_0$. At the same time we show that if $\lambda$ is chosen
appropriately then the excess of $M$ with respect to $T_\lambda$ at
scale $R$ grows at least at rate $R^{\gamma-1 + \epsilon_0}$ for some
$\epsilon_0 > 0$. This uses that fact that $r^\gamma$ on $\bC$ is generated by pushing into the $T_\lambda$, and is the smallest
possible growth rate of admissible Jacobi fields on $\bC$. Combining
these two results we get a contradiction, unless the excess of $M$
with respect to $T_\lambda$ is zero, i.e. $M=T_\lambda$. 

\textbf{Acknowledgements} N.E. was supported in part by NSF grant
DMS-2204301. G. Sz. was supported in part by NSF grant DMS-2203218.  We thank Otis Chodosh for helpful conversations.

\section{Preliminaries} \label{sec:background}

Throughout this paper $\bC_0$ will be a smooth minimal hypercone in $\R^{n+1}$, $l$ a non-negative integer, and $\bC = \bC_0 \times \R^l \subset \R^{n+l+1} \equiv \R^{n+1} \times \R^l = \{ (x, y) : x \in \R^{n+1} , y \in \R^l\}$.  When we write $u : \bC \to \R$ we mean $u : \reg \bC \to \R$.  Define $B_\rho(\xi)$ to be the open Euclidean ball in $\R^{n+l+1}$ of radius $\rho$ centered at $\xi$, $B_\rho \equiv B_\rho(0)$, and $A_{r, \rho} = B_r \setminus \overline{B_\rho}$ to be the open annulus centered at $0$.  Write $\omega_{n}$ for the volume of the Euclidean $n$-ball.  Let $\eta_{X, \rho}(Y) = (Y - X)/\rho$ be the translation/rescaling map.

\subsection{Cylindrical cones}
The Jacobi operator on $\bC_0$ is $L_{\bC_0} f = \Delta_{\bC_0} f + |A_{\bC_0}|^2 f$.  In polar coordinates $x = r\theta$ this becomes
\[
L_{\bC_0} = \del_r^2 + r^{-1} (n-1)\del_r + r^{-2} L, \quad L = \Delta_\Sigma + |A_\Sigma|^2,
\]
so that $L_\Sigma = L + (n-1) = \Delta_\Sigma + |A_\Sigma|^2 + (n-1)$ is the Jacobi operator of $\Sigma \subset S^n$.  Write $\lambda_1 < \lambda_2 \leq \lambda_3 \leq \ldots$ for the eigenvalues of $L$, and $\{ \psi_j \}_{j \geq 1}$ for the corresponding $L^2(\Sigma)$-ON basis eigenfunctions, so that $L \psi_j + \lambda_j \psi_j = 0$. Define
\[
\gamma_j^\pm = -(n-2)/2 \pm \sqrt{((n-2)/2)^2 + \lambda_j},
\]
so that every linear combination $u(x = r\theta) = c_j^+ r^{\gamma_j^+} \psi_j(\theta) + c_j^- r^{\gamma_j^-} \psi_j(\theta)$ is a Jacobi field on $\bC_0$.  We assume $\bC_0$ is \emph{strictly stable}, which means that $\gamma_j^- < -(n-2)/2 < \gamma_j^+$.  For shorthand we shall write $\gamma_j = \gamma_j^+$ and $\gamma = \gamma_1 = \gamma_1^+$.

Let $H_\pm$ be leaves of the Hardt-Simon foliation~\cite{HS85} of $\bC_0$, lying on different sides of $\bC_0$, so that each $H_\pm$ is oriented compatibly with $\bC_0$ (i.e. so that $\nu_{H_\pm} \to \nu_{\bC_0}$ as $r \to \infty$).  We assume $\bC_0$ is \emph{strictly minimizing}, which means there is a radius $R_0(\bC_0)$ so that (possibly after appropriately rescaling $H_\pm$)
\begin{equation}\label{eqn:H-graph}
H_\pm\setminus B_{R_0} = \graph_{\bC_0}(\Psi_\pm),
\end{equation}
where
\begin{equation}\label{eqn:H-graph2}
\Psi_\pm(x = r\theta) = \pm r^{\gamma} \psi_1(\theta) + v_\pm, \quad |v_\pm| \leq r^{\gamma - \alpha_0}
\end{equation}
for some $\alpha_0(\bC_0) > 0$ (see e.g. \cite[Equation (10),
p. 114]{HS85}).  It follows by standard elliptic estimates
(see e.g. \cite[Proposition 2.2]{Sz21}) that 
\begin{gather}\label{eqn:v-bounds}
|\nabla^i v_\pm| \leq c(\bC, i) r^{\gamma - i - \alpha_0}.
\end{gather}

\vspace{3mm}

Define
\[
H(t) = \left\{ \begin{array}{l l} |t|^{1/(1-\gamma)} H_{\mathrm{sign}(t)} & t \neq 0 \\ \bC_0 & t = 0 \end{array} \right. ,
\]
so that
\begin{gather}\label{eqn:def-Psi}
H(t) \setminus B_{|t|^{1/(1-\gamma)}R_0} = \graph_{\bC_0}(\Psi_t), \quad \Psi_t(x) = |t|^{1/(1-\gamma)} \Psi_{\mathrm{sign}(t)}(|t|^{-1/(1-\gamma)}x),
\end{gather}
and hence
\[
|\Psi_t(x) - t r^{\gamma} \psi_1(\theta) | \leq
|t|^{1+\alpha_0/(1-\gamma)} r^{\gamma - \alpha_0}. 
\]

\begin{lemma}\label{lem:Phi}
For sufficiently small $\eps$ (depending only on $\bC_0$), we can write $(1+\eps)H_+$ as a graph over $H_+$ of the function $\Phi_{+, \eps}$, which we can expand as
\begin{equation}\label{eqn:Phi-concl1}
\Phi_{\eps, +} = \eps \Phi_+ + \eps^2 V_{+, \eps},
\end{equation}
where: $\Phi_+$ is a positive Jacobi field on $H_+$ satisfying
\begin{gather}
\Phi_+(x+\Psi_+(x)\nu_{\bC_0}(x)) = (1-\gamma) r^\gamma \psi_1(\theta) + O(r^{\gamma-\alpha_0}) \text{ for } x = r\theta \in \bC_0 \setminus B_{R_0}, \\
\text{ and }|\nabla^i \Phi_+(x)| \leq c(\bC, i) |x|^{\gamma - i}, \quad i = 0, 1, 2, \ldots ;
\end{gather}
and $V_{+, \eps}$ satisfies the estimates
\begin{equation}
|\del_\eps^j \nabla^i V_{+,\eps}(x)| \leq c(\bC, i, j) |x|^{\gamma - i}.
\end{equation}
The same statements hold with $(1-\eps)H_-$, $\Phi_{-, \eps}$, $\Phi_-$, $V_{-, \eps}$ in place of $(1+\eps)H_+$, $\Phi_{+,\eps}$, $\Phi_+$, $V_{+, \eps}$.
\end{lemma}

\begin{proof}
The decomposition \eqref{eqn:Phi-concl1} simply follows from the definition of Jacobi field.  Positivity of $\Psi_+$ comes from the star-shapedness of $H_+$. For $x = r\theta$ with $r \gg 1$, we can write
\[
\Phi_{\eps, +}(x + \Psi_+(x) \nu_{\bC_0}(x)) (1+E_1(x)) = \left(  (1+\eps)\Psi_+(E_2(x)/(1+\eps)) - \Psi_+(E_2(x)) \right) .
\]
where $E_1, E_2 - id$ are smooth functions which are (at minimum) linearly controlled by $r^{-1} \Psi_+(x), \nabla \Psi_+(x)$.  We compute:
\begin{align*}
&(1+\eps)\Psi_+(x/(1+\eps)) - \Psi_+(x) \\
&= ((1+\eps)^{1-\gamma} - 1)r^\gamma \psi_1(\theta) + \int_1^{1+\eps} (v_+ - r \del_r v_+)|_{x/\lambda} d\lambda \\
&= ((1-\gamma) \eps + O(\eps^2))r^\gamma \psi_1(\theta) + O(\eps r^{\gamma-\alpha_0}).
\end{align*}
The bounds for $\Psi_+$ and $\nabla^i V_{+,\eps}$ follow by the above computations and standard elliptic estimates.
\end{proof}

\subsection{Minimal surfaces and varifolds} 

It will be convenient to use the language of varifolds, see \cite{SimonGMT} for a standard references.  We shall write $||V||$ for the mass measure of a varifold, and given a countably-$(n+l)$-rectifiable set $M \subset \R^{n+l+1}$ we write $[M]$ for the integral $(n+l)$-varifold with mass measure $\haus^{n+l} \llcorner M$.

Recall that the monotonicity formula for stationary (integral)
$(n+l)$-varifolds in $\R^{n+l+1}$ says the density ratio 
\[
\theta_V(\xi, \rho) := \frac{||V||(B_\rho(\xi))}{\omega_{n+l} \rho^{n+l}}
\]
is increasing in $\rho$, for any $\xi \in \R^{n+l+1}$, and is constant if and only if $V$ is a cone over $\xi$.  We define the density of $V$ at a point $\xi$, resp. and at $\infty$, by
\[
\theta_V(\xi) = \lim_{\rho \to 0} \theta_V(\xi, \rho), \quad \text{resp. } \theta_V(\infty) = \lim_{\rho \to \infty} \theta_V(0, \rho).
\]
If $V = [M]$, we understand $||M||(U) \equiv ||V||(U)$, $\theta_M(\xi, \rho) \equiv \theta_{[M]}(\xi, \rho)$, etc.

\begin{lemma}\label{lem:cone-liou}
Let $V$ be a non-zero stationary integral varifold cone in $\R^{n+l+1}$, such that $\spt V$ lies to one side of $\bC$.  Then $V = k[\bC]$ for some integer $k \geq 1$.
\end{lemma}

\begin{proof}
Follows by the maximum principles of \cite{SolomonWhite}, \cite{ilmanen}, since $\sing \bC$ has dimension at most $n+l-7$.
\end{proof}

\subsection{$\beta$-harmonic functions and Jacobi fields}

For $\beta > 0$, \cite{SimonLiouville} introduced the notion of $\beta$-harmonic functions, which are functions $h(r, y)$ on $B_1^+ \subset \R^{1+l}_+ = \{ (r, y) \in \R \times \R^l : r > 0 \}$ solving
\begin{equation}\label{eqn:beta-eqn}
r^{-1-\beta} \del_r (r^{1+\beta} \del_r h) + \Delta_y h = 0,
\end{equation}
and satisfying the integrability hypothesis
\begin{equation}\label{eqn:beta-bound}
\int_{B_1^+} r^{-2} h^2  r^{1+\beta} < \infty.
\end{equation}
\cite{SimonLiouville} showed any such $h$ extends analytically in $r^2$ and $y$ to $\{ (r, y) : r^2 + |y|^2 < 1, r \geq 0\}$, and in particular can be written as a sum of homogenous $\beta$-harmonic polynomials in $r^2, y$.

In spherical coordinates $(r, y) = \rho \omega$, where $\rho = \sqrt{r^2 + |y|^2}$ and $\omega = (r, y)/\rho$, \eqref{eqn:beta-eqn} becomes
\begin{equation}\label{eqn:beta-sphere}
\rho^{-l-1-\beta} \del_\rho( \rho^{l+1+\beta} \del_\rho h) + \rho^{-2} \omega_1^{-1-\beta} \mdiv_{S^l}(\omega_1^{1+\beta} \nabla_{S^l} h) = 0.
\end{equation}
Here $\omega_1 = \omega \cdot \del_r \equiv r/\sqrt{r^2 + |y|^2}$.

\cite{SimonLiouville} showed that the $\beta$-harmonic homogenous polynomials $\{ h_q \}$, when restricted to $S^l_+$, are $L^2(\omega_1^{1+\beta} d\omega)$-complete.  So there is an $L^2(\omega_1^{1+\beta} d\omega)$-ON basis of functions $\{\phi_i\}_{i \geq 1}$, each being the restriction of a $\beta$-harmonic homogenous polynomial $h_i(\rho\omega) = \rho^{q_i} \phi_i(\omega)$.  From \eqref{eqn:beta-sphere}, we get the eigenvalue-type equation
\begin{equation}
\omega_1^{-1-\beta} \mdiv_{S^l}(\omega_1^{1+\beta} \nabla_{S^l} \phi_i) + q_i(q_i+l+\beta) \phi_i = 0
\end{equation}
on $S^l_+$.

\vspace{3mm}

For each $j \geq 1$, define $\beta_j = n-2 + 2\gamma_j = 2\sqrt{((n-2)/2)^2 + \lambda_j}$.  Let $v$ be a Jacobi field on $\bC \cap B_1$ satisfying
\[
\int_{\bC \cap B_1} |x|^{-2} v^2 < \infty.
\]
For every $j$, let $h_j(r, y) = r^{-\gamma_j} \int_\Sigma v(r\theta, y) \phi_j(\theta) d\theta$.  Then a straightoforward computation shows each $h_j$ is $\beta_j$-harmonic in $\bC \cap B_1$, and hence admits an analytic expansion of the form
\[
h_j = \sum_{i \geq 1} h_{ij}(r, y), 
\]
where each $h_{ij}$ is a $q_{ij}$-homogenous $\beta_j$-harmonic polynomial, for some integer $q_{ij} \geq 0$.  Moreover, all the $\{ h_{ij}|_{S^l_+} \}_i$ are $L^2(\omega_1^{1+\beta_j} S^l_+)$-orthogonal.

Therefore $v$ admits an expansion
\begin{equation}\label{eqn:fourier0}
v(r\theta, y) = \sum_{i, j \geq 1} r^{\gamma_j} \psi_j(\theta) h_{ij}(r, y) ,
\end{equation}
which holds in the following senses: in $L^2(\Sigma)$ for every fixed $(r, y)$; in $L^2(\bC \cap B_\rho)$ for every $\rho < 1$; in $C^\infty_{loc}(\bC \cap B_1 \setminus \{r = 0 \})$.  For every $0 < \rho < 1$ we have
\begin{equation}\label{eqn:fourier}
\int_{\bC \cap B_\rho} v^2 = \sum_{i, j \geq 0} c_{ij}^2 \rho^{n+l+2\gamma_j+q_{ij}} = \sum_{i \geq 1} a_i^2 \rho^{n+l+2q_i},
\end{equation}
where $\gamma_1 = q_1 < q_2 < \ldots$.  Note \eqref{eqn:fourier} implies that the function
\begin{equation}\label{eqn:fourier-inc}
\rho \mapsto \rho^{-n-l-2\gamma} \int_{\bC \cap B_\rho} v^2
\end{equation}
is increasing in $\rho$.

\vspace{3mm}

We require a few helper lemmas about ``tame'' Jacobi fields.
\begin{lemma}[\cite{Sz20}]\label{lem:linfty-l2}
Let $v$ be a Jacobi field on $\bC \cap B_1$ with $\sup_{\bC \cap B_1} ||x|^{-\gamma} v| < \infty$.  Then for every $\theta < 1$ we have the estimate
\begin{equation}\label{eqn:linfty-l2-concl1}
\sup_{\bC \cap B_\theta} ||x|^{-\gamma} v|^2 \leq c(\bC, \theta) \int_{\bC \cap B_1} v^2,
\end{equation}
and
\begin{equation}\label{eqn:linfty-l2-concl2}
\int_{\bC \cap B_1} v^2 \leq \int_{\bC \cap B_1} |x|^{-2} v^2 \leq c(\bC) \sup_{\bC \cap B_1} ||x|^{-\gamma} v|^2.
\end{equation}
\end{lemma}

\begin{proof}
We prove \eqref{eqn:linfty-l2-concl1} for $\theta = 1/8$, and the statement for general $\theta$ will follow by standard elliptic estimates and an obvious covering argument.  Pick $(x, y) = (r\theta, y) \in \bC \cap B_{1/8}$ with $r > 0$.  By scale-invariant elliptic estimates and \eqref{eqn:fourier-inc}, we compute
\begin{align*}
|v(r\theta, y)|^2 
&\leq c(\bC) r^{-n-l} \int_{\bC \cap B_{r/2}(x, y)} v^2 \\
&\leq c(\bC) r^{-n-l} \int_{\bC \cap B_r(0, y)} v^2 \\
&\leq c(\bC) r^{2\gamma} \int_{\bC \cap B_{1/4}(0, y)} v^2 \\
&\leq c(\bC) r^{2\gamma} \int_{\bC \cap B_{1/2}} v^2.
\end{align*}
To prove \eqref{eqn:linfty-l2-concl2} simply use $-(n-2)/2 < \gamma < 0$.
\end{proof}

\begin{lemma}\label{lem:linear-liouville}
Let $v$ be a non-negative Jacobi field on $\bC$, with $\sup_{\bC \cap B_R} ||x|^{-\gamma} v| < \infty$ for all $R$.  Then $v = a |x|^\gamma \psi_1(\theta)$, for some constant $a$.
\end{lemma}

\begin{proof}
We first note that if $h(r, y)$ is $\beta$-harmonic in $B_R \subset
\R^{l+1}_+$, then by (e.g.) integrating \eqref{eqn:beta-sphere}, we have the mean-value equality
\begin{equation}\label{eqn:ll-1}
h(0, 0) = \left( \int_{B_\rho^+} r^{1+\beta} drdy \right)^{-1} \int_{B_\rho^+} h(r, y) r^{1+\beta} dr dy
\end{equation}
for every $0 < \rho < R$.  Therefore, if $h$ is $\beta$-harmonic and non-negative in $\R^{l+1}_+$, we have $h(0, y) \leq h(0, y')$ for every $y, y' \in \R^l$, and hence $h(0, y) \equiv h(0, 0)$.  (In fact $h$ must be entirely constant, but we will not directly need this.)  In particular, for $h_j(r, y)$ as in \eqref{eqn:fourier0}, 
\begin{equation}\label{eqn:ll-2}
h_1(0, y) = h_1(0, 0) \quad \forall y \in \R^l.
\end{equation}

Now for every $(r, y)$, we have $v(r\theta, y) = \sum_j r^{\gamma_j} h_j(r, y) \psi_j(\theta)$ in $L^2(\Sigma)$, and hence in $L^1(\Sigma)$.  Using \eqref{eqn:ll-1} we compute for every $\rho > 0$:
\begin{align}\label{eqn:ll-3}
\int_{\bC \cap B_\rho} |x|^\gamma v \psi_1 = \int_{B_\rho^+} h_1(r, y) r^{n-1+2\gamma} dr dy = c(\bC) \rho^{n+l+2\gamma} h_1(0, 0).
\end{align}
Since $1/c(\bC) \leq \psi_1 \leq c(\bC)$ and $v \geq 0$, we can use \eqref{eqn:ll-2}, \eqref{eqn:ll-3}, and standard elliptic estimates at scale $r/2$ to deduce
\begin{align}
r^{-\gamma} v(r\theta, y) 
&\leq c r^{-n-l-\gamma} \int_{B_{r/2}(r\theta, y)} v \\
&\leq c r^{-n-l-2\gamma} \int_{B_{r/2}(r\theta, y)} |x|^\gamma v \psi_1 \\
&\leq c r^{-n-l-2\gamma} \int_{B_r(0, y)} |x|^\gamma v \psi_1 
\leq c h_1(0, y) = c(\bC) h_1(0, 0). \label{eqn:ll-4}
\end{align}

If $v(r\theta, y) \neq a r^\gamma \psi_1(\theta)$ for some constant $a$, then from the expansion \eqref{eqn:fourier0}, \eqref{eqn:fourier} and equation \eqref{eqn:linfty-l2-concl2} we must have
\[
\rho^{2\alpha}/C \leq \rho^{-n-l-2\gamma} \int_{\bC \cap B_\rho} v^2 \leq c(\bC) \sup_{\bC \cap B_\rho} |r^{-\gamma} v|^2
\]
for some $\alpha > 0$ and some constant $C$ independent of $\rho$.  For $\rho \gg 1$ this contradicts \eqref{eqn:ll-4}.
\end{proof}

\vspace{3mm}

Lastly, we will require the following ``baby'' 3-annulus-type lemma.
\begin{lemma}\label{lem:1d-3ann}
Let $\{q_i \in \R\}_{i \in \Z}$ be an increasing sequence, and $\{b_i \in \R\}_{i \in \Z}$ be arbitrary.  Fix $k \in \Z$, and $3\eps \in (0, q_{k+1} - q_k)$, and $T \geq 1/\eps$.  Define $\psi(t) = \sum_i b_i^2 e^{2q_i t}$.  Then 
\[
\psi(t+T) \geq e^{2(q_k+\eps)T} \psi(t) \implies \psi(t+2T) \geq e^{2(q_{k+1}-\eps)T} \psi(t+T).
\]
\end{lemma}

\begin{proof}
By replacing $b_i^2$ with $b_i^2 e^{2q_i t}$, it suffices to take $t = 0$.  Observe that the first inequality implies
\[
\sum_{i \geq k+1} b_i^2 e^{2 q_i T} \geq (e^{2\eps T} - 1) \sum_{i \leq k} b_i^2 e^{2 q_i T}.
\]
Then we get
\begin{align*}
\psi(2T) &= \sum_{i} b_i^2 e^{4 q_i T}  \\
& \geq e^{2 (q_{k+1} - \eps) T} \sum_{i \geq k+1} b_i^2 e^{2 q_i T} + e^{2( q_{k+1} - \eps) T}(e^{2\eps T} - 1) \sum_{i \geq k+1} b_i^2 e^{2 q_i T} \\
&\geq e^{2 (q_{k+1} - \eps) T} \left( \sum_{i \geq k+1} b_i^2 e^{2 q_i T} + (e^{2\eps T} - 1)^2 \sum_{i \leq k} b_i^2 e^{2 q_i T} \right) \\
&\geq e^{2(q_{k+1} - \eps) T} \psi(T). \qedhere
\end{align*}
\end{proof}

\section{Barriers}

In this section we collate our functions and hypersurfaces we will use
as barriers.
We write $L_{H_\pm}$, $L_{H(\lambda)}$ for the Jacobi operator on $H_\pm$, $H(\lambda)$.

\begin{lemma}[{\cite[Proposition 2.8]{Sz21}}]\label{lem:F_a}
For any $a > \gamma$, there are functions $F_{\pm, a}$ on $H_\pm$ satisfying
\begin{gather}
F_{\pm, a}(x + \Psi_\pm(x)\nu_{\bC_0}(x)) = r^a \phi_1(\theta) \text{ for $x = r\theta \in \bC$ and $r \gg 1$}, \\
|\nabla^i F_{\pm, a}| \leq c(\bC_0, a, i) r^{a-i}, \quad \text{ and } L_{H_\pm} F_{a, \pm} \geq c(\bC_0, a)^{-1} r^{a-2}.
\end{gather}

We extend $F_{\pm a}$ to an smooth $a$-homogenous function $F_a : \R^{n+1} \setminus \{0 \} \to \R$ by setting
\[
F_a(x) = \left\{ \begin{array}{l l} \lambda^a F_{+, a}(\lambda^{-1} x) & x \in \lambda H_+ \\ r^a\phi_1 & x \in \bC_0 \setminus \{0 \} \\ \lambda^a F_{-, a}(\lambda^{-1} x) & x \in \lambda H_- \end{array} \right. , \quad \lambda > 0.
\]
Each $F_a$ satisfies
\begin{equation}\label{eqn:F_a-props}
|D^i F_a| \leq c(\bC_0, a, i) |x|^{a - i}, \quad L_{H(\lambda)}(F_a|_{H(\lambda)}) \geq |x|^{a-2}/c(\bC_0, a) .
\end{equation}
\end{lemma}

\begin{proof}
See \cite[Proposition 2.8]{Sz21} and \cite[Lemma 5.7]{Sz20}.
\end{proof}

Write $\cM_{T_\lambda}(G)$ for the mean curvature of the graph of $G$ over $T_\lambda$, which is well-defined provided $|x|^{-1} G$ is sufficiently small.  For $G, H \in C^2(T_\lambda)$, write $D\cM_{T_\lambda}(G)[H]$ for the linearization of $\cM_{T_\lambda}$ at $G$ in the direction $H$, that is
\[
D\cM_{T_\lambda}(G)[H] = \frac{d}{dt} \Big|_{t = 0} \cM_{T_\lambda} (G + t H).
\]
Provided $|x|^{-1} G$, $\nabla G$, $|x| \nabla^ 2G$ are sufficiently small then $D \cM_{T_\lambda}(G)$ is a linear elliptic operator on $C^2(T_\lambda) \to C^0(T_\lambda)$.  Of course if $G = 0$ then $D\cM_{T_\lambda}(0) \equiv L_{T_\lambda}$ is the Jacobi operator on $T_\lambda$.

\begin{lemma}\label{lem:L_G}
For $\gamma' > \gamma$, there are constants $\eps(\bC, \gamma')$, $c(\bC, \gamma')$ so that if $G : T_1 \to \R$ is a $C^2$ function satisfying $|\nabla^i G| \leq \eps |x|^{\gamma - i}$ for $i = 0, 1, 2$, then
\[
D\cM_{T_1}(G)[F_{\gamma'}|_{T_1}]|_{(x, y)} \geq \frac{1}{c} |x|^{\gamma'-2} > 0,
\]
where $F_{\gamma'}$ is from Lemma \ref{lem:F_a}.  The same result also holds for $T_{-1}$ in place of $T_1$.
\end{lemma}

\begin{proof}
  Let us recall that the $C^3$-regularity scale $r_{C_3}(M, x)$ of a hypersurface
  $M\subset \mathbf{R}^N$ at a point $x\in M$ is defined to be the
  supremum of those $r
  > 0$ for which the translated and rescaled surface $r^{-1}(M-x)$ is
  the graph of a $C^3$ 
  function $u$ inside the unit ball, with $|u|_{C^3} \leq 1$.

The $C^3$ regularity scale of $T_1$ satisfies $|x|/c \leq r_{C_3}(T_1,
x, y) \leq c |x|$, and $|x| \geq 1/c$ on $T_1$, for $c = c(\bC)$. It
follows that
provided $\eps(\bC)$ is sufficiently small, at any point $(x, y) =
(r\theta, y) \in T_1$ we can write 
\[
\cM_{T_1}(G) = L_{T_1} G + r^{-1} R(x, r^{-1} G, \nabla G, r \nabla^2 G)
\]
where $R(x, z, p, q)$ is a smooth, uniformly bounded function which is quadratically controlled by $z, p, q$.  More precisely, we have the bounds
\[
|R(x, z, p, q)| \leq c (|z|^2 + |p|^2 + |q|^2) , \quad |\del_z R| + |\del_p R| + |\del_q R| \leq c (|z| + |p| + |q|).
\]

For any $H \in C^2(T_1)$, we have
\[
D\cM_{T_1}(G)[H] = \frac{d}{ds} \Big|_{s = 0} \cM_{T_1}(G + s H) = L_{T_1} H + a_{ij} \nabla^2_{ij} H + r^{-1} b_i \nabla_i H + r^{-2} c H
\]
with $a_{ij}, b_i, c$ continuous functions satisfying
\[
|a_{ij}| + |b_i| + |c| \leq c \eps r^{\gamma-1}.
\]
Taking $H = F_{\gamma'}|_{T_1}$ and using Lemma \ref{lem:F_a}, we get
\[
D\cM_{T_1}(G)[F_{\gamma'}] \geq (\frac{1}{c(\bC_0, \gamma')} - c \eps r^{\gamma-1})r^{\gamma'-2} \geq \frac{1}{2 c}r^{\gamma'-2}
\]
provided $\eps(\bC,\gamma')$ is chosen sufficiently small.  The argument for $T_{-1}$ is verbatim.
\end{proof}

\vspace{3mm}

We say a set $A \subset \R^{n+l+1}$ lies above (resp. below) $H(\lambda)\times \R^l$ if $A \subset \cup_{\mu \geq \lambda} H(\mu)$ (resp. $A \subset \cup_{\mu \leq \lambda} H(\mu)$).  More generally, if $U \subset \R^{n+1+l}$ and $S \subset U$ divides $U$ into two disjoint connected components $U_\pm$, and there are $\lambda_- < \lambda_+$ such that
\begin{align}
U_+ \cap (H(\lambda_+) \times \R^l) \neq \emptyset = U_+ \cap (H(\lambda_-) \times \R^l) \\
U_- \cap (H(\lambda_-) \times \R^l) \neq \emptyset = U_- \cap (H(\lambda_+) \times \R^l),
\end{align}
then we say $A \subset U$ lies above $S$ in $U$ (resp. below $S$ in $U$) if $A \subset \overline{U_+}$ (resp. $A \subset \overline{U_-}$).

If $S$ is a smooth hypersurface of $U$, and $\overline{S}$ divides $U$ into components $U_\pm$ as in the previous paragraph, we say $S$ has positive (resp. negative) mean curvature if the mean curvature vector of $S \cap U$ never vanishes and always points into $U_+$ (resp. into $U_-$).

\begin{theorem}[{\cite[Proposition 2.9]{Sz21}}]\label{thm:barrier}
There is a large odd integer $p$ and a constant $Q > 0$ depending only on $\bC$ so that the following holds.  Let $I$ be an open set in $\R^l$, and let $f : I \to \R$ be a $C^3$ function satisfying $|f|_{C^3(I)} \leq K$ for some $K > Q$.  Then for any $\eps < 1/Q$ there is a closed (as sets) oriented hypersurface $X_\eps$ in $\{ 0 < |x| < K^{-Q^2}, y \in I \}$ satisfying:
\begin{enumerate}
\item $X_\eps$ is $C^2$ with negative mean curvature;
\item at any point $(0, y) \in \overline{X_\eps} \cap \{ |x| = 0 , y \in I \}$, the tangent cone of $X_\eps$ at $(0, y)$ is the graph of $-\eps |x|$ over $\bC$;
\item $X_\eps$ varies continuously (in the Hausdorff distance) with $\eps$, and for every $y' \in I$ the $y$-slice $X_\eps \cap \{ y = y' \}$ is trapped between
\begin{gather}
H(\eps f(y')^p - \eps) \text{ and } H(\eps f(y')^p + \eps).
\end{gather} 
\end{enumerate}
In particular, if $V$ is a stationary varifold in $U \subset \{ y \in I, |x| < K^{-Q^2} \}$ which lies below $X_\eps$ in $U$, then $\spt V \cap X_\eps \cap U = \emptyset$.
\end{theorem}

\begin{proof}
This is proved with $l = 1$ in \cite[Proposition 2.9]{Sz21}.  When $l > 1$, the functions $f(y)$, $G(y)$, $E(y)$ (as defined in the proof \cite{Sz21}) becomes functions on $\R^{l}$, and so bounds on $G^{(i)}$ become bounds on $D^i G$, and the Jacobi operator $L_{H(\lambda)\times\R^l}$ on $H(\lambda) \times \R^l$ becomes $\Delta_y + L_{H(\lambda)}$.  Otherwise the same proof carries over with only cosmetic changes.
\end{proof}

We shall also need the following computation.
\begin{lemma}\label{lem:tilde-S}
Let $S$ be a $C^2$ hypersurface in $\R^{n+1}$, and $f : \R^l \to \R_+$ a $C^2$ function.  Define the new hypersurface $\tilde S \subset \R^{n+1+l}$ by
\[
\tilde S = \bigcup_{y \in \R^l} (f(y) S) \times \{y\}.
\]
At any point $z = ( f(y) x, y) \in \tilde S$, let $\nu$ be a choice of normal for $S$ at $x$, and let $\tilde \nu$ be the normal of $\tilde S$ at $z$ pointing in the same direction as $\nu$.  Then the mean curvature $\cM_{\tilde S}$ of $\tilde S$ with respect to $\tilde \nu$ at $z$ can be expressed as
\begin{align*}
\cM_{\tilde S} &= \frac{1}{\sqrt{E}}\left[ \frac{\cM_S}{f} + \frac{|Df|^2 h_S(x^T, x^T)}{f E} + (x\cdot \nu)\left(-\delta_{\alpha \beta} + \frac{(x\cdot \nu)^2 D_\alpha f D_\beta f}{E} \right) D^2_{\alpha \beta} f \right].
\end{align*}
where $E = 1+|Df|^2 (x \cdot \nu)^2$, and $\cM_S$ is the mean curvature of $S$, and $h_S$ is the second fundamental form of $S$.
\end{lemma}

\begin{proof}
Let $F(x^1, \ldots, x^n) : U \subset \R^n \to \R^{n+1}$ be a coordinate chart for $S$.  Let $g_{ij} = \del_i F \cdot \del_j F$ be the induced metric in these coordinates, and $h_{ij} = - \del^2_{ij} F \cdot \nu$ the second fundamental form.  WLOG let us assume $g_{ij}(0) = \delta_{ij}$.

Define coordinate chart $\tilde F(x^1, \ldots, x^n, y^1, \ldots, y^l) : U \times \R^l \to \R^{n+l+1}$ for $\tilde S$ by $\tilde F(x, y) = (f(y) F(x), y)$.  We have
\[
\del_i \tilde F = (f \del_i F, 0), \quad \del_\alpha \tilde F = ( (\del_\alpha f) F, e_\alpha)
\]
where we abbreviate $\del_i \equiv \frac{\del}{\del x^i}$, $\del_\alpha \equiv \frac{\del}{\del y^\alpha}$, and write $\{e_\alpha\}$ for the standard basis vectors of $\R^l$.  The metric $\tilde g$ in these coordinates at $(x, y) = (0, 0)$ is therefore
\[
\tilde g_{\alpha \beta} = \delta_{\alpha \beta} + (\del_\alpha f) (\del_\beta f) |F|^2, \quad \tilde g_{\alpha i} = f(\del_\alpha f) (F \cdot \del_i F) , \quad \tilde g_{ij} = f^2 \delta_{ij}.
\]
One can verify directly that the metric inverse at $(x, y) = (0, 0)$ is then
\begin{gather*}
\tilde g^{\alpha \beta} = \delta_{\alpha \beta} - \frac{(F \cdot \nu)^2}{E} (\del_\alpha f) (\del_\beta f) , \quad \tilde g^{\alpha i} = -\frac{ (\del_\alpha f) (F \cdot \del_iF)}{f E} \\
\tilde g^{ij} = f^{-2} \delta_{ij} + \frac{ |D f|^2 (F \cdot \del_i F) (F \cdot \del_jF)}{f^2 E} 
\end{gather*}
where $E = 1+|Df|^2 (F \cdot \nu)^2$.  

We have
\[
\del_i \del_j \tilde F = (f \del_i \del_j F, 0), \quad \del_i \del_\alpha \tilde F = (\del_\alpha f \del_i F, 0), \quad \del_\alpha \del_\beta \tilde F = (\del_\alpha \del_\beta f F, 0).
\]
Therefore, since by inspection $\tilde \nu = E^{-1/2}(\nu, - (\del_\alpha f)(F\cdot \nu) e_\alpha)$ we can compute the second fundamental form of $\tilde S$ to be
\[
\tilde h_{\alpha \beta} = \frac{-(\del^2_{\alpha \beta} f) (F \cdot \nu)}{\sqrt{E}}, \quad \tilde h_{\alpha i} = 0, \quad \tilde h_{ij} = \frac{ f h_{ij}}{\sqrt{E}}.
\]

We deduce that, at $(x, y) = (0, 0)$, we have 
\begin{align*}
\cM_{\tilde S}
&= \tilde g^{\alpha \beta} \tilde h_{\alpha \beta} + \tilde g^{ij} \tilde h_{ij} \\
&= \left(\delta_{\alpha \beta} - \frac{(F\cdot \nu)^2 D_\alpha f D_\beta f}{E}\right)\left(\frac{-D^2_{\alpha \beta} f (F\cdot \nu)}{\sqrt{E}} \right) \\
&\quad + f^{-2} \left(\delta_{ij} + \frac{|Df|^2 (F \cdot \del_i F) (F\cdot \del_j F)}{E}\right) \frac{f h_{ij}}{\sqrt{E}} 
\end{align*}
which, recalling that $g_{ij} = \delta_{ij}$ at $(0, 0)$ is the form required.
\end{proof}

\section{Non-concentration}

For shorthand let us write $T_\lambda = H(\lambda) \times \R^l$, so we
also have $T_0 = \bC$. We
define the following notion of ``distance from $T_\lambda$''.  This is
effectively a non-linear version of the norm $\sup_U ||x|^{-\gamma} u|$,
see Corollary \ref{cor:nc}. 
\begin{definition}
Given subsets $M, U \subset B_1$, $\lambda \in \R$, define $D_{T_\lambda}(M; U)$ as the least $d \geq 0$ such that $M \cap U$ is trapped between $H(\lambda \pm d) \times \R^l$.
\end{definition}

The following follows directly from the definition:
\begin{equation}\label{eq:triangle}
  D_{T_\lambda}(M; U) \leq D_{T_{\lambda'}}(M; U) + |\lambda -
  \lambda'|. 
\end{equation}

Note that $D$ scales like $\rho^{1-\gamma}$ in the sense that for $c >
0$ we have
\[ D_{c T_\lambda}(cM ; cU) = c^{1-\gamma} D_{T_\lambda}(M; U). \]
  We define a scale-invariant ``excess'' quantity which will be our
  main mechanism for measuring decay/growth.
\begin{definition}
Given $R > 0$, $\lambda \in \R$, and subset $M \subset B_R$, define the excess of $M$ in $B_R$ w.r.t. $T_\lambda$ to be
\[
E(M, T_\lambda, R) = D_{R^{-1} T_\lambda}(R^{-1} M; B_1) \equiv R^{\gamma - 1}D_{T_\lambda}(M; B_R).
\]
\end{definition}
We remark that $E(T_\lambda, \bC, R) = R^{\gamma-1}|\lambda|$ and that
for $M$ that is the graph of $u$ over $T_\lambda$ we can think of
$E(M, T_\lambda, R)$ as equivalent to 
$\sup_{B_R} R^{\gamma-1} \big| |x|^{-\gamma} u\big|$.

\vspace{3mm}

\begin{lemma}\label{lem:small-H-graph}
If $d = D_{T_\lambda}(M; U) \leq \beta |\lambda|$ for some $\beta(\bC)$ sufficiently small, then $M \cap U$ is trapped between the graphs $\graph_{T_\lambda}(\pm c(\bC) d |x|^\gamma)$.  Conversely, if $M \cap U$ is trapped between $\graph_{T_\lambda}(\pm d |x|^\gamma)$ and $d \leq \beta |\lambda|$, then $D_{T_\lambda}(M; U) \leq c(\bC) d$.
\end{lemma}

\begin{proof}
By scale-invariance it suffices to consider the case when $\lambda = \pm 1$, in which case the Lemma follows straightforwardly from Lemma \ref{lem:Phi}.
\end{proof}

\vspace{3mm}

The main Theorem of this Section is the following non-concentration result.  We emphasize that $c_0$ in \eqref{eqn:nc-concl1}, \eqref{eqn:nc-concl2} is independent of $s$.

\begin{theorem}[Non-concentration]\label{thm:nc}
Given any $s \in (0, 1/2]$ and $\theta \in (0, 1)$, there are constants $c_0(\bC, \theta)$, $r_0(\bC, \theta, s)$, $\delta_0(\bC, \theta, s)$ so that the following holds.  Let $M$ be a complete minimal hypersurface in $B_1$, such that $D_{T_\lambda}(M; B_1) < \delta_0$ for $|\lambda| < \delta_0$, and $M \cap B_1 \cap \{ r \geq r_0\}$ is trapped between $\graph_{T_\lambda}(\pm b |x|^\gamma)$ for $b < \delta_0$.  Then
\begin{equation}\label{eqn:nc-concl1}
D_{T_\lambda}(M; B_{\theta}) \leq c_0(b + s D_{T_\lambda}(M; B_1))
\end{equation}
If $B_1$ is replaced by $A_{1,\rho}$ in our assumptions, for some $\rho \in (0, 1/2]$, then instead we get
\begin{equation}\label{eqn:nc-concl2}
D_{T_\lambda}(M; A_{\theta, \theta^{-1}\rho}) \leq c_0(b + s D_{T_\lambda}(M; A_{1, \rho})).
\end{equation}
(with $c_0, r_0, \delta_0$ depending on $\rho$ also).
\end{theorem}

\begin{proof}[Proof of Theorem \ref{thm:nc}] For ease of notation
  write $d = D_{T_\lambda}(M; B_1)$.  We need to break the proof into
  two cases, depending on whether $d \gtrsim |\lambda|$ (when $M$ is
  about as close to $\bC$ as it is to $T_\lambda$), or whether $d <<
  |\lambda|$ (when $M$ is much closer to $T_\lambda$ than to $\bC$).
  In the first case we will use the barrier surfaces constructed in
  Theorem \ref{thm:barrier}.  In the second case we will construct
  barrier surfaces as graphs over $T_\lambda$.  At the end of the proof
  we will explain the (very minor) changes required to get
  \eqref{eqn:nc-concl2}. 

Fix $\gamma < \gamma' < \min\{ \gamma+1/2, 0\}$.  Throughout the proof
\[
1/2 \geq \beta(\bC, \theta, \gamma') \gg r_0(\bC, \theta, \beta, \gamma', s) \gg \delta_0(\bC, \theta, \beta, \gamma', s, r_0)
\]
are small constants which we shall choose as we proceed, but can a posteriori be fixed. 

We first claim that $D_{T_\lambda}(M; B_1 \cap \{ |x| \geq r_0 \})
\leq c(\bC, \beta) b$.  If $b \leq \beta |\lambda|$ this follows from
Lemma \ref{lem:small-H-graph}, provided $\beta(\bC)$ is sufficiently
small.  Suppose now $b > \beta |\lambda|$.  Then provided
$\delta_0(\bC, r_0)$ is sufficiently small,  $M \cap B_1 \cap \{ |x|
\geq r_0\}$ is trapped between the graphs of $\pm c(\bC) (b + |\lambda|) |x|^\gamma$ over $\bC \cap \{ |x| \geq r_0/2\}$ in $\{ |x| \geq r_0/2\}$, and hence trapped between the graphs of $\pm c(\bC) (b/\beta) |x|^\gamma$ over $\bC \cap \{ |x| \geq r_0/2\}$.  But then provided $\delta_0(\bC, r_0, \beta)$ is sufficiently small, $\graph_{\bC \cap \{ |x| \geq r_0/2\}}( c(\bC) (b/\beta) |x|^\gamma)$ is trapped between $H(\pm c(\bC) b/\beta) \times \R^l$ in $\{ |x| \geq r_0/2\}$.  Combined with the inequality $|\lambda| \leq b/\beta$, our initial claim follows.

We shall henceforth work towards proving the estimate
\begin{equation}\label{eqn:nc-1}
D_{T_\lambda}(M; \{ |x| \leq r_0 \} \cap \{ |y| \leq \theta^2 \}) \leq c(\bC, \beta, \theta)(b + s d).
\end{equation}
Provided $r_0(\bC)$ is sufficiently small, \eqref{eqn:nc-1} combined with our initial claim will imply \eqref{eqn:nc-concl1} (with $\theta^2$ in place of $\theta$).  For ease of notation let us define the domains
\[
\Omega_1 = \{ |x| \leq r_0 \} \cap \{ |y| \leq \theta-s \}, \quad \Omega_2 = \{ |x| \leq r_0 \} \cap \{ |y| \leq \theta^2 \}.
\]
We now break into two cases as outlined at the start of the proof.

\vspace{3mm}

\textbf{Case 1:} $\bar d := d + s^{-1} b > \beta |\lambda|$.  Here we use the barriers constructed in Theorem \ref{thm:barrier}.  Note first that $\bar d \leq 2 \delta_0/s$, so by ensuring $\delta_0(\bar d, s)$ is small, we can assume $\bar d$ is small also.

Fix $p$ as in Theorem \ref{thm:barrier}, and fix $\sigma : \R \to \R$ a smooth function satisfying $|\sigma(z)^p - z| \leq 1/10$.  Define
\begin{equation}\label{eqn:nc-2}
f(y) = \sigma(t^{-1} \lambda + h(y)), \quad h(y) = (\theta - |y|)^{-1}.
\end{equation}
Note that on $\Omega_1$ we have $h \geq 4$ and $|D^k h| \leq c(k, s, \theta)$.

Provided $t \geq s |\lambda|$, we have $|f|_{C^3(\Omega_1)} \leq
c(\bC, s, \theta)$.  Therefore there are $t_0(\bC)$, $r_0(\bC, s,
\theta)$ so that for every $s |\lambda| \leq t \leq t_0$, there are
surfaces $X_t$ defined in $\Omega_1$ with negative mean curvature, as
constructed in Theorem \ref{thm:barrier}.  Each $y$-slice $X_t \cap \{
y = y' \}$ is trapped between $H(t f(y')^p \pm t) \times \{y'\}$ in
$\R^{n+1} \times \{y'\}$, and hence (recalling our definition of
$\sigma, f$) is trapped between $H(\lambda + 2 t h(y'))\times \{y'\}$
and $H(\lambda + t h(y')/2) \times \{y'\}$. 

Provided $\delta_0(\bC)$ is sufficiently small, we have $\lambda + t_0 h(y')/2 \geq t_0/4$.  Therefore since $\bar d \geq d$, taking $\delta_0(\bC)$ smaller as necessary, we deduce $M$ lies below $X_{t_0}$ in $\Omega_1$.  Set $t_1 = \beta^{-2} (b + sd) = \beta^{-2} s \bar d$.  We claim that $M$ lies below $X_{t_1}$ in $\del\Omega_1$, provided $\beta(\bC)$ is chosen sufficiently small.

Let $S_1 = \del\Omega_1 \cap \{ |y| = \theta - s \}$.  In $S_1$ we have $h \geq 1/s$, and so $X_{t_1}$ lies above $H(\lambda + \beta^{-2} \bar d/2) \times \R^l$ in $S_1$.  But of course $\beta^{-2} \bar d/2 \geq d$, and so $H(\lambda + \beta^{-2} \bar d/2) \times \R^l$ lies above $M$ in $S_1$.

Let $S_2 = \del \Omega_1 \cap \{ |x| = r_0\}$.  In $S_2$, $X_{t_1}$ lies above $H(\lambda + \beta^{-2} \bar d) \times \R^l$, and hence above $H(\beta^{-2} \bar d/2) \times \R^l$.  On the other hand, provided $\bar d(\beta)$, $\delta_0(\beta, \bar d, r_0)$ are sufficiently small, in $S_2$ $M$ lies below $\graph_{T_\lambda}(b |x|^\gamma)$, which in $S_2$ lies below $\graph_\bC( c \beta^{-1}\bar d |x|^\gamma)$, which in $S_2$ lies below $H( c \beta^{-1} \bar d) \times \R^l$.  Our claim follows by ensuring $\beta(\bC)$ is small.

Since $t \mapsto X_t$ is continuous in the Hausdorff distance, by Theorem \ref{thm:barrier} and the previous claim we can bring $t$ from $t_0$ down to $t_1$ to deduce $M$ lies below $X_{t_1}$ in $\Omega_1$.  In particular, since on $\Omega_2$ we have $h \leq c(\theta)$, we deduce that each $y$-slice of $M \cap \Omega_2$ lies below $H(\lambda + c(\bC, \beta, \theta) (b + sd))$.  Repeating the above argument with the orientations reversed implies that $M \cap \Omega_2$ is trapped between $H(\lambda \pm c(\bC, \beta, \theta)(b + sd))$.  This proves Case 1.

\vspace{3mm}

\textbf{Case 2:} $d + s^{-1} b \leq \beta |\lambda|$.  Here we construct graphical barriers for $M$ over $T_\lambda$.  There is no loss in generality in assuming $\lambda > 0$.  For ease of notation write $\Phi(x, \eps) = \Phi_{\eps, +}(x)$ for the graphing function of $(1+\eps)H_+$ over $H_+$ as in Lemma \ref{lem:Phi}, and set $\mu = |\lambda|^{1/(1-\gamma)}$.  Define for $i = 1, 2$ the domains $\tilde \Omega_i = \{ (x, y) : (x/2, y) \in \Omega_i\}$.

First note that on $T_\lambda$ we have the inequality $|x| \geq \mu/c(\bC)$.  Second, recall that the $C^3$ regularity scale of $T_\lambda$ at $x$ is comparable to $|x|$.  Third, note that by Lemma \ref{lem:small-H-graph} (ensuring $\beta(\bC)$ is small) we know that
\begin{equation}\label{eqn:nc-2.5}
\text{$M$ is trapped between the graphs $\graph_{T_\lambda}( \pm c_1 d |x|^\gamma)$ in $B_1$},
\end{equation}
for some constant $c_1(\bC)$, and hence $M$ is trapped between $\graph_{T_\lambda}( \pm c_1 \beta |\lambda| |x|^\gamma)$ in $B_1$.

For $A(\bC, \gamma')$ a large constant to be determined later, let $\eta(t) : \R \to \R$ be a smooth increasing function satisfying $\eta(t) = t$ for $|t| < A \beta/2$, $\eta(t) \equiv \sign(t) A \beta$ for $|t| \geq A \beta$, and $|\eta'| \leq 10$.  For $t \in [0, 1]$ define $G_t(x, y) : T_\lambda \cap \{ |y| < \theta \} \to \R$ by
\begin{gather}\label{eqn:nc-3}
G_t(x, y) = \mu \Phi( \mu^{-1} x, \eta(t h(y)) ) - \eta(t) |\lambda| F_{\gamma'}(x),
\end{gather}
where $F_{\gamma'}$ as in Lemma \ref{lem:F_a}.  The $G_t$ will define our graphical barriers.

From Lemma \ref{lem:Phi}, for $(x, y) = (r\theta, y) \in T_\lambda \cap \tilde\Omega_1$ we have
\[
\mu \Phi(\mu^{-1} x, \eta(th(y))) = \eta(t h(y)) (\mu \Phi_+(\mu^{-1} x) \pm c(\bC) A \beta |\lambda| r^{\gamma}),
\]
and so, ensuring $\beta(A, \bC)$ is small, we get
\[
\eta(th(y)) |\lambda| r^{\gamma}/c \leq \mu \Phi(\mu^{-1} x \eta(th(y))) \leq c\eta(th(y)) |\lambda| r^{\gamma}.
\]
for $c = c(\bC)$.  Since $|F_{\gamma'}(x)| \leq c |x|^{\gamma'}$ and $\eta(th(y)) \geq \eta(t)$, ensuring $r_0(\gamma', \bC)$ is small, we deduce that
\begin{gather}\label{eqn:nc-4}
\eta(th(y)) |\lambda| |x|^{\gamma}/c \leq G_t(x, y) \leq c\eta(th(y)) |\lambda| |x|^{\gamma} \leq c A \beta |\lambda| |x|^{\gamma}
\end{gather}
on $T_\lambda \cap \tilde \Omega_1$, for $c = c(\bC)$.

By a similar computation, recalling that $\gamma' > \gamma$ and $|x| = r$, we have
\begin{align*}
|\nabla G_t(x, y)| 
&\leq c \eta |\lambda| r^{\gamma-1} + c t |Dh| |\lambda| r^{\gamma} + c \eta |\lambda| r^{\gamma'-1} \\
&\leq (c A \beta + c t |Dh| r) |\lambda| r^{\gamma-1},
\end{align*}
and
\begin{align*}
|\nabla^2 G_t(x, y)| \leq (c A \beta + c |Dh| r + c |D^2 h| r^2) |\lambda| r^{\gamma-2},
\end{align*}
where $c = c(\bC, \gamma')$.  Ensuring $r_0(\bC, \gamma', \beta, s, \theta)$ is sufficiently small, and recalling the bound $|x| \geq \mu/c(\bC)$ on $T_\lambda$, we get for $i = 0, 1, 2$ the bounds
\begin{gather}\label{eqn:nc-5}
|\nabla^i G_t(x, y)| \leq c(\bC, \gamma') A \beta |\lambda| |x|^{\gamma-i} \leq c(\bC, \gamma') A \beta |x|^{1-i} \quad \text{ on } T_\lambda \cap \tilde \Omega_1.
\end{gather}
In particular, ensuring $\beta(\bC, \gamma', A)$ is sufficiently small we get that $\graph_{T_\lambda}(G_t)$ is a smooth hypersurface without boundary in $\Omega_1$.

We aim to show the graph of $G_t$ has negative mean curvature in $\Omega_1$.  We first compute
\[
\cM_{T_\lambda}(G_t) = \cM_{T_\lambda}( \mu \Phi(\mu^{-1} x, \eta(t h(y)))) - \eta(t) |\lambda| \int_0^1 D \cM_{T_\lambda}(G_{t, s}) [ F_{\gamma'}] ds =: I + II
\]
where $G_{t, s}(x, y) = \mu \Phi(\mu^{-1} x \eta(t h(y))) - s \eta(t)
|\lambda| F_{\gamma'}(x)$.  We claim that, at $(x, y)  \in T_\lambda
\cap \tilde \Omega_1$, with $|x|=r$, we have
\begin{gather}\label{eqn:nc-6}
|I| \leq c ( t^2 |Dh|^2 + t|D^2 h|) |\lambda| r^{\gamma}, \quad II \leq - \eta(t) |\lambda| r^{\gamma'-2}/c
\end{gather}
for $c = c(\bC, \gamma')$.  Bounds \eqref{eqn:nc-6} will imply that on $T_\lambda \cap \tilde\Omega_1$ and for $0 < t \leq A \beta$ we have
\begin{gather}\label{eqn:nc-6.5}
\cM_{T_\lambda}(G_t) \leq c(\bC, \gamma') t |\lambda| (r |h|_{C^2(\tilde\Omega_1)} - 1)r^{\gamma'-2} < 0,
\end{gather}
provided we ensure $r_0(\bC, \gamma', \theta, s)$ is chosen sufficiently small.

Let us prove our claim for $|I|$, i.e. the first inequality in \eqref{eqn:nc-6}.  By construction, $\graph_{T_\lambda} (\mu \Phi(\mu^{-1} x, \eta(t h(y)))) \cap \{ |y| < \theta \}$ coincides with $S \cap \{ |y| < \theta\}$ where $S$ is the hypersurface
\[
S = \bigcup_{|y| < \theta} \left[ (1+\eta(th(y))) H(\lambda) \right] \times \{y\}.
\]
Since $\max\{\eta(th(y))|x|,  \mu\Phi(\mu^{-1} x, \eta(th(y)))\} \leq |x|/2$ provided $\beta(\bC, A)$ is sufficiently small, it will suffice to prove the bound
\begin{gather}\label{eqn:nc-7}
|\cM_S( (1+\eta(th(y)))x, y)| \leq c (t^2 |Dh|^2 + t |D^2 h|) |\lambda| |x|^{\gamma}
\end{gather}
for any $((1+\eta(th(y)))x, y) \in S \cap \{ |y| < \theta\}$, where $\cM_S$ is the mean curvature of $S$.  From Lemma \ref{lem:tilde-S}, for the same $x, y$ as above we have the bound
\begin{gather}\label{eqn:nc-8}
|\cM_S| \leq c(l)|x \cdot \nu_{H(\lambda)}(x)| |D^2 \eta(t h(y))| + c(l)|D \eta(t h(y))|^2 |h_{H(\lambda)}(x^T, x^T)|,
\end{gather}
where $h_{H(\lambda)}$ is the second fundamental form of $H(\lambda)$, and $\nu_{H(\lambda)}$ the unit normal.  Trivially we have
\begin{gather}
|D \eta(t h(y))|^2 \leq c t^2 |Dh|^2, \quad |D^2 \eta(t h(y))| \leq c t |D^2 h| + c t^2 |Dh|^2, \label{eqn:nc-9} \\
\text{and} \quad |h_{H(\lambda)}(x^T, x^T)| \leq c |x|, \label{eqn:nc-10}
\end{gather}
for $c = c(\bC)$.

If $|x| \leq R_0 \mu$ (for $R_0$ as in \eqref{eqn:H-graph}) then since $\mu/c(\bC) \leq |x|$ also, the bound \eqref{eqn:nc-7} follows from \eqref{eqn:nc-8}, \eqref{eqn:nc-9}, \eqref{eqn:nc-10} and the inequality $|x| \leq R_0 |\lambda| |x|^{\gamma}$.  If $|x| > R_0 \mu$, then near $x$, $H(\lambda)$ is graphical over $\bC_0$ by the function $\Psi_\lambda$ as in \eqref{eqn:def-Psi}.  From \eqref{eqn:H-graph2}, \eqref{eqn:v-bounds} we have
\begin{gather}\label{eqn:nc-11}
|x \cdot \nu_{H(\lambda)}(x)| \leq c |\nabla \Psi_\lambda| |x| \leq c |\lambda| |x|^{\gamma},  \quad |h_{H(\lambda)}(x^T, x^T)| \leq c |\nabla^2 \Psi_\lambda| |x|^2 \leq c |\lambda| |x|^{\gamma},
\end{gather}
and the bound \eqref{eqn:nc-7} follows from \eqref{eqn:nc-8}, \eqref{eqn:nc-9}, \eqref{eqn:nc-11}.

We consider now the bound for $II$.  By similar computations as before, we have
\[
|\nabla^i G_{t, s}(x, y)| \leq c(\bC, \gamma') A \beta |\lambda| |x|^{\gamma-i}, \quad (i = 0, 1, 2), \quad \text{ on } T_\lambda \cap \tilde\Omega_1,
\]
for any $s, t \in [0, 1]$.  By scaling and the definition of $F_{\gamma'}$ we have
\begin{align}
D\cM_{T_\lambda}(G_{t, s}) [F_{\gamma'}]|_{(x, y)} &= \mu^{-2} D\cM_{T_1}(G_{t, s}^\mu)[F_{\gamma'}(\mu \cdot)]|_{(\mu^{-1} x, y)} \\
&= \mu^{\gamma'-2} D\cM_{T_1}(G^\mu_{t, s})[F_{\gamma'}]|_{(\mu^{-1} x, y)},
\end{align}
where $G_{t, s}^\mu(\xi, \zeta) = \mu^{-1} G_{t, s}(\mu \xi, \zeta)$.  Using \eqref{eqn:nc-5} we have on $T_1 \cap \mu^{-1} \tilde\Omega_1$, 
\[
|\nabla^i G^\mu_{t, s}| \leq \mu^{-1+\gamma} c A \beta |\lambda| |x|^{\gamma-i} \leq c A\beta |x|^{\gamma - i} \quad (i = 0, 1, 2),
\]
and therefore provided $\beta(\bC,\gamma', A)$ is sufficiently small, we can apply Lemma \ref{lem:L_G} to deduce
\[
\mu^{\gamma'-2} D\cM_{T_1}(G^\mu_{t, s})[F_{\gamma}] |_{(\mu^{-1} x, y)} \geq |x|^{\gamma'-2}/c.
\]
This proves the bound for $II$ in \eqref{eqn:nc-6}, completing the proof of our claim and hence the inequality \eqref{eqn:nc-6.5}.

We now use $G_t$ to control $D_{T_\lambda}(M; \Omega_2)$.  First note that if $0 < t \leq A \beta$, then \eqref{eqn:nc-4} implies
\begin{gather}\label{eqn:nc-12}
G_t(x, y) \geq \min\{  t h(y), A \beta \} |\lambda| |x|^{\gamma}/c \geq \min\{ t, A \beta\} |\lambda| |x|^{\gamma}/c \quad \text{ on } T_\lambda \cap \tilde\Omega_1,
\end{gather}
for $c = c(\bC, \gamma')$.  Therefore, by ensuring $A(\bC, \gamma')$ is sufficiently large from \eqref{eqn:nc-2.5} we know that $M$ lies below $\graph_{T_\lambda}(G_{A \beta})$ in $\Omega_1$.  Set $t_1 = \min\{ \beta^{-2} |\lambda|^{-1} (b + sd), A \beta \}$.  We claim that, provided $\beta(\bC)$ is chosen sufficiently small, $M$ lies below $\graph_{T_\lambda}(G_t)$ in $\del\Omega_1$ for every $t_1 \leq t \leq A \beta$.

We prove this claim.  Of course if $t_1 = A \beta$ there is nothing to show, so let us assume $t_1 < A \beta$.  Suppose $(x, y) + G_t(x, y)\nu_{T_\lambda}(x, y) \in \del\Omega_1 \cap S_1$.  Then $|y| = \theta - s$ and (by \eqref{eqn:nc-5}) $|x| < 2r_0$.  Since $h(y) \geq 1/s$, we can use \eqref{eqn:nc-12} and our assumption $d + s^{-1} b \leq \beta |\lambda|$ to estimate
\[
G_t(x, y) \geq \min\{ (\beta^{-2} |\lambda|^{-1} s d) s^{-1} |\lambda| |x|^\gamma/c, A \beta|\lambda| |x|^\gamma/c \} \geq c_1 d |x|^\gamma
\]
provided $\beta(\bC)$, $A(\bC)^{-1}$ are chosen sufficiently small.  Therefore $\graph_{T_\lambda}(G_t)$ lies above $\graph_{T_\lambda}(c_1 d |x|^{\gamma})$ in $S_1$, and hence lies above $M$ in $S_1$.

Suppose $(x, y) + G_t(x, y) \nu_{T_\lambda}(x, y) \in \del\Omega_1 \cap S_2$.  Then as before $(x, y) \in \tilde\Omega_1$, and we can estimate instead
\[
G_{t}(x, y) \geq \min\{ (\beta^{-2} |\lambda|^{-1} b) |\lambda| |x|^{\gamma} /c, A \beta |\lambda| |x|^\gamma/c\}  \geq b |x|^{\gamma},
\]
again ensuring $\beta(\bC), A(\bC)^{-1}$ are small.  We deduce $\graph_{T_\lambda}(G_t)$ lies above $\graph_{T_\lambda}(b |x|^\gamma)$ in $S_2$, and hence by our assumptions lies above $M$ in $S_2$.  This finishes the proof of our claim.

By our last two claims and the negative mean curvature \eqref{eqn:nc-6.5} we can bring $t$ from $A \beta$ down to $t_1$ and deduce by the maximum principle \cite{SolomonWhite} that $M$ lies below $\graph_{T_\lambda}(G_t)$ in $\Omega_1$.  In particular, since $h \leq c(\bC, \theta)$ on $\tilde \Omega_2$ from \eqref{eqn:nc-4} we get that $M$ lies below $\graph_{T_\lambda}( c(\bC, \theta, \gamma') (b + sd) |x|^\gamma)$ in $\Omega_2$.  Repeating the argument with the orientation swapped, we deduce $M$ is trapped between $\graph_{T_\lambda}(\pm c(\bC, \theta, \gamma') (b + sd) |x|^\gamma)$ in $\Omega_2$.  Since $b + sd \leq s \beta |\lambda|$, ensuring $\beta(\bC, \theta, \gamma')$ is sufficiently small we can apply apply Lemma \ref{lem:small-H-graph} to finish the proof of Case 2.

\vspace{3mm}

\textbf{With $A_{1,\rho}$ in place of $B_1$:} To get \eqref{eqn:nc-concl2}, we only need to modify our definition of $h, \Omega_1, \Omega_2$.  In this case, we define
\[
h(y) = (|y| - \theta^{-1} \rho)^{-1} + (\theta - |y|)^{-1},
\]
and
\[
\Omega_1 = \{ r \leq r_0 \} \cap \{ |y| \in [(\theta^{-1}+s)\rho, \theta-s] \}, \quad \Omega_2 = \{ r \leq r_0 \} \cap \{ |y| \in [\theta^{-2}\rho, \theta^2] \}.
\]
The proof for \eqref{eqn:nc-concl2} is then verbatim to the proof above for \eqref{eqn:nc-concl1}, of course replacing $B_1$ with $A_{1,\rho}$ wherever it occurs, and allowing all constants to depend on $\rho$ also.
\end{proof}

The main utility of Theorem \ref{thm:nc} is in the below Corollary
\ref{cor:nc} concerning inhomogeneous  blow-up limits, in particular in the lower bound of Item \ref{item:nc3}.
\begin{cor}\label{cor:nc}
Let $M_i$ be a sequence of complete minimal hypersurfaces in $B_1$, and $\lambda_i \to 0$.  Suppose that
\[
D_{T_{\lambda_i}}(M_i; B_1) \to 0, \quad (1/2)||\bC||(B_1) \leq ||M_i||(B_1) \leq (3/2) ||\bC||(B_1),
\]
and let $\mu_i$ be a sequence such that $\sup_i \mu_i^{-1} D_{T_{\lambda_i}}(M_i; B_1) < \infty$.

Then, first, there are $\tau_i \to 0$ so that
\[
M_i \cap B_{1-\tau_i} \cap \{ |x| \geq \tau_i \} = \graph_{T_{\lambda_i}}(u_i),
\]
for $u_i : B_{1-\tau_i/2} \cap \{ |x| \geq \tau_i/2\} \to \R$ smooth functions satisfying
\[
|x|^{-1} |u_i| + |\nabla u_i| + |x| |\nabla^2 u_i| \leq \tau_i.
\]
Second, passing to a subsequence, we can find a Jacobi field $v$ on $\bC \cap B_1$ so that for any given $\theta < 1$ we have:
\begin{enumerate}
\item \label{item:nc1} $\mu_i^{-1} u_i \to v$ smoothly on compact subsets of $\bC \cap B_1\cap \{ |x| > 0 \}$;

\item \label{item:nc2} $\sup_{\bC \cap B_1} ||x|^{-\gamma} v| \leq c(\bC) \liminf_i \mu_i^{-1} D_{T_{\lambda_i}}(M_i; B_1)$;

\item \label{item:nc3} $\limsup_i \mu_i^{-1} D_{T_{\lambda_i}}(M_i; B_{\theta^2}) \leq c(\bC, \theta) \sup_{\bC \cap B_{\theta}} ||x|^{-\gamma} v|$.
\end{enumerate}
Third, given any $\rho \in (0, 1/2]$, the above Corollary also holds with $A_{1,\rho}$, $A_{1-\tau_i, \rho+\tau_i}$, $A_{1-\tau_i/2, \rho+\tau_i/2}$, $A_{\theta, \theta^{-1} \rho}$, $A_{\theta^2, \theta^{-2} \rho}$ in place of $B_1$, $B_{1-\tau_i}$, $B_{1-\tau_i/2}$, $B_{\theta}$, $B_{\theta^2}$ (resp.), in which case all constants depend on $\rho$ also.
\end{cor}

\begin{remark}
Since $2\gamma > -n+2$, and by Lemma \ref{lem:linfty-l2}, we have for every $\theta < 1$:
\[
\frac{1}{c(\bC, \theta)} \sup_{\bC \cap B_{\theta}} ||x|^{-\gamma} v|^2 \leq \int_{\bC \cap B_1} |v|^2 \leq \int_{\bC \cap B_1} |x|^{-2} |v|^2 \leq c(\bC) \sup_{\bC \cap B_1} ||x|^{-\gamma} v|^2.
\]
\end{remark}

\begin{proof}
The existence of $\tau_i, u_i$ follows from the definition of $D$, the constancy theorem, and Allard's theorem by a standard argument.  For convenience write $U_i = B_{1-\tau_i} \cap \{ |x| > \tau_i\}$ and $d_i = D_{T_{\lambda_i}}(M_i; B_1)$.  After passing to a subsequence we can assume that $\Gamma = \lim_i \mu^{-1}_i d_i$ exists, and for all $i$, either $d_i > \beta|\lambda_i|$ or $d_i \leq \beta |\lambda_i|$, for $\beta$ a small number to be determined momentarily.

By definition of $D$, for all $i$, $M_i \cap B_1$ is trapped
between $H(\lambda_i \pm d_i) \times \R^l$.  If $d_i >
\beta|\lambda_i|$, then by ensuring $\tau_i \to 0$ sufficiently
slowly, from equations \eqref{eqn:H-graph}, \eqref{eqn:H-graph2} we get that $M_i \cap B_1 \cap \{ |x| > \tau_i\}$ is trapped in $\{ |x| > \tau_i\}$ between $\graph_{\bC \cap \{ |x| > \tau_i/2\}}(\pm c(\bC)(d_i + |\lambda_i|) |x|^\gamma)$, and hence $|u_i| \leq c (d_i + |\lambda_i) |x|^\gamma \leq c(\bC, \beta)d_i |x|^\gamma$.  If $d_i \leq \beta|\lambda_i|$, then provided $\beta(\bC)$ is sufficiently small Lemma \ref{lem:small-H-graph} implies $M_i \cap B_1$ is trapped between $\graph_{T_\lambda}(\pm c(\bC) d_i |x|^\gamma)$, and hence $|u_i| \leq c(\bC) d_i |x|^\gamma$.

Either way, we have that
\begin{equation}\label{eqn:cornc-1}
\sup_{T_{\lambda_i} \cap U_i} ||x|^{-\gamma} u_i| \leq c(\bC, \beta) d_i,
\end{equation}
and hence by standard elliptic theory we can pass to a subsequence,
find a Jacobi field $v$ on $\bC \cap B_1$, and get smooth convergence
$\mu^{-1} u_i \to v$ on compact subsets of $\bC \cap B_1 \cap \{ |x| >
0 \}$.  The estimate \eqref{eqn:cornc-1} implies 
\[
\sup_{\bC \cap U} ||x|^{-\gamma} v| \leq c(\bC, \beta) \Gamma \quad \forall U \subset\subset \bC \cap B_1 \cap \{ |x| > 0\},
\]
which proves Items \ref{item:nc1}, \ref{item:nc2}.

To prove Item \ref{item:nc3}, we use Theorem \ref{thm:nc} and our hypotheses, to deduce that for every $s > 0$ there is an $r_0 > 0$ so that for $i \gg 1$ we have
\[
D_{T_{\lambda_i}}(M_i; B_{\theta^2}) \leq c_0 \sup_{T_{\lambda_i} \cap B_{\theta} \cap \{ |x| > r_0 \}} ||x|^{-\gamma} u_i| + c_0 s D_{T_{\lambda_i}}(M_i; B_1),
\]
where $c_0 = c_0(\bC, \theta)$ is independent of $s$.  We can
therefore take a limit as $i \to \infty$, and then as $s \to 0$, we
deduce Item \ref{item:nc3}.

\end{proof}

\section{Geometric 3-annulus lemma}

\begin{lemma}\label{lem:3ann}
Given $\eps < \eps_0(\bC)/16$, we can find an $R_0(\bC, \eps) > 1$ so
that for every $R \geq R_0$, there is a $\delta_0(\bC, \eps, R) > 0$
so that the following holds.

Let $|\lambda| < \delta_0$, and let $M$ be a complete minimal hypersurface in $B_R$, such that
\begin{equation}\label{eqn:3ann-hyp}
E(M, \bC, R) < \delta_0, \quad \theta_M(0, R) \leq (3/2) \theta_\bC(0).
\end{equation}
Then:
\begin{align}
&E(M, T_\lambda, 1) \geq E(M, T_\lambda, 1/R) R^{\gamma-1+\eps} \\
&\implies E(M, T_\lambda, R) \geq E(M, T_\lambda, 1) R^{\gamma - 1 +\eps_0 - \eps}. \label{eqn:3ann-concl}
\end{align}
\end{lemma}

\begin{proof}
Set $\eps_0 = \min\{ q_2 - q_1, 1 \}$ for $q_i$ as in \eqref{eqn:fourier}.  Assume $R_0 \geq e^{2/\eps}$, so that we can write $R = R_*^k$ for some integer $k \geq 1$ and some $R_* \in [e^{2/\eps}, e^{4/\eps})$.  We will show the Lemma holds provided $k(\bC, \eps)$ (and hence $R_0)$ is sufficiently large, to be determined below.

Suppose the Lemma failed.  Then we have sequences $\delta_i \to 0$, $\lambda_i \to 0$, and complete minimal hypersurfaces $M_i$ in $B_R$ so that \eqref{eqn:3ann-hyp} holds but
\begin{align*}
&E(M_i, T_{\lambda_i}, 1/R) \leq E(M_i, T_{\lambda_i}, 1) R^{-\gamma+1-\eps} \\
&\text{ and } E(M_i, T_{\lambda_i}, R) \leq E(M_i, T_{\lambda_i}, 1) R^{\gamma - 1 + \eps_0 - \eps}.
\end{align*}

Since \eqref{eqn:3ann-concl} vacuously holds if $E(M, T_{\lambda}, 1) = 0$, there is no loss in assuming $M_i \cap B_1 \neq \emptyset$ for all $i$.  Then by our hypotheses \eqref{eqn:3ann-hyp}, standard compactness of stationary integral varifolds, and the constancy theorem, we deduce $M_i \to [\bC]$ as varifolds in $B_{R}$.  By Allard's theorem we can find an exhaustion $U_i$ of $B_R \setminus \{ |x| = 0 \}$ so that
\[
M_i \cap U_i = \graph_{T_{\lambda_i}}(u_i)
\]
for smooth functions $u_i$.  By Corollary \ref{cor:nc}, after passing to a subsequence, the rescaled functions $E(M_i, T_{\lambda_i}, 1)^{-1} u_i$ converge on compact subsets of $\bC \cap B_R \setminus \{ r = 0 \}$ to a Jacobi field $v$ on $\bC \cap B_R$ satisfying
\begin{align*}
&\sup_{\bC \cap B_{1/R}} (1/R)^{\gamma-1} ||x|^{-\gamma} v| \leq c(\bC) R^{-\gamma+1-\eps}, \\
&\sup_{\bC \cap B_2} ||x|^{-\gamma} v| \geq 1/c(\bC), \\
&\sup_{\bC \cap B_R} R^{\gamma-1} ||x|^{-\gamma} v| \leq c(\bC) R^{\gamma-1+\eps_0-\eps}.
\end{align*}

Define
\[
S(i)^2 = R_*^{i(-n-l)} \int_{B_{R_*^i}} v^2.
\]
Then from Lemma \ref{lem:linfty-l2} we have
\begin{align}
&S(-k) \leq c (1/R)^\gamma \sup_{B_{1/R}} ||x|^{-\gamma} v| \leq c R_*^{-k(\gamma+\eps)} \label{eqn:3ann-0.1}\\
&S(1) \geq c^{-1} R_*^{-n-l} \sup_{B_2} ||x|^{-\gamma} v| \geq 1/c(\bC, \eps) \label{eqn:3ann-0.2}\\
&S(k) \leq c R_*^{k(\gamma+\eps_0-\eps)}. \label{eqn:3ann-0.3}
\end{align}

We claim that, for any $\eta > 0$, provided $k(\bC, \eps, \eta) \in \N$ is chosen sufficiently large, then we have $S(1) \leq \eta$, which will contradict \eqref{eqn:3ann-0.2} for $\eta(\bC, \eps)$ sufficiently small.  We prove this claim.  First assume
\begin{equation}\label{eqn:3ann-1}
S(1) \geq R_*^{(k+1)(\gamma+\eps/2)} S(-k).
\end{equation}
Then by Lemma \ref{lem:1d-3ann} and our choice of $R_*$, we have
\[
S(k) \geq R_*^{(k-1)(\gamma+\eps_0-\eps/2)} S(1),
\]
which implies
\begin{equation}\label{eqn:3ann-2}
S(1) \leq c(\bC, \eps) R_*^{-k\eps/2} \leq \eta,
\end{equation}
provided we ensure $k(\bC, \eps, \eta)$ is large.  On the other hand, if \eqref{eqn:3ann-1} fails, then we again have \eqref{eqn:3ann-2} (for perhaps a larger constant $c(\bC, \eps)$, and hence a larger $k(\bC, \eps, \eta)$).  This proves our claim, and finishes the proof of Lemma \ref{lem:3ann}.
\end{proof}

\section{Growth of entire hypersurfaces}

\begin{prop}\label{prop:decay-once}
There are constants $\eps_1(\bC)$, $\rho_1(\bC)$ so that for every
$\rho \leq \rho_1$, and $\eta > 0$, we can find a $\delta_1(\bC, \eta,
\rho)$ satisfying the following. Let $M$ be a complete minimal hypersurface in $B_1$ satisfying
\begin{equation}\label{eqn:decay-once-hyp}
E(M, \bC, 1) < \delta_1, \quad ||M||(B_1) \leq (3/2)||\bC||(B_1). 
\end{equation}
Then we can find a $\lambda \in (-\eta, \eta)$ so that
\begin{equation}\label{eqn:decay-once-concl}
E(M, T_\lambda, \rho) \leq \rho^{\gamma - 1 + \eps_1}E(M, T_\lambda, 1),
\end{equation}
and $E(M, T_\lambda, 1) \leq c(\bC) E(M, \bC, 1)$.
\end{prop}

\begin{proof}
Suppose the Proposition failed.  Then for $\eps_1, \rho_1$ to be determined later, we can find sequences $\delta_i \to 0$, and complete minimal hypersurfaces $M_i$ in $B_1$ satisfying \eqref{eqn:decay-once-hyp} but failing \eqref{eqn:decay-once-concl} for all $\lambda \in (-\eta, \eta)$.  Let $\lambda_i'$ minimize $\lambda \mapsto E(M_i, T_\lambda, 1)$.  Trivially $\lambda_i' \to 0$ and $E(M_i, T_{\lambda_i'}, 1) \leq E(M_i, \bC, 1)$.

By standard compactness and the constancy theorem, $M_i \to \kappa[\bC]$ as varifolds for $\kappa \in \{0, 1\}$.  Since \eqref{eqn:decay-once-concl} is trivially satisfied if $M \cap B_\rho = \emptyset$, by our contradiction hypothesis we have $M_i \cap B_\rho \neq \emptyset$ for all $i$, and hence $\kappa = 1$.  Allard's theorem implies therefore $M_i \to \bC$ smoothly on compact subsets of $B_1 \setminus \{ |x| = 0 \}$.

For $U_i$ an exhaustion of $B_1 \setminus \{ |x| = 0 \}$, we can write $M_i \cap U_i = \graph_{T_{\lambda_i'}}(u_i')$.  Passing to a subsequence, by Corollary \ref{cor:nc} we can get convergence $E(M, T_{\lambda_i'}, 1)^{-1} u_i' \to v$, for some Jacobi field on $\bC \cap B_1$ with $\sup_{\bC \cap B_1} ||x|^{-\gamma} v| \leq c(\bC)$, and hence (by Lemma \ref{lem:linfty-l2}) $\int_{\bC \cap B_1} |x|^{-2} v^2 \leq c(\bC)$.

By \eqref{eqn:fourier0}, \eqref{eqn:fourier}, we can write $v(r\theta, y) = a r^\gamma \psi_1(\theta) + z(r\theta, y)$, for $|a| \leq c(\bC)$ and $z$ satisfying the decay
\[
\rho^{-n-l} \int_{\bC \cap B_\rho} z^2 \leq c(\bC) \rho^{2\gamma+4\eps_1}
\]
for some $\eps_1(\bC) > 0$ determined by the spectral decomposition of $\bC_0$.  Using Lemma \ref{lem:linfty-l2} we deduce
\begin{equation}\label{eqn:decay-once-1}
\rho^\gamma \sup_{\bC \cap B_\rho} ||x|^{-\gamma} z| \leq c(\bC) \left(
  \rho^{-n-l}\int_{\bC \cap B_{2\rho}} z^2 \right)^{1/2} \leq c(\bC) \rho^{\gamma+2\eps_1}
\end{equation}
for all $\rho \leq 1/2$.

Let $\lambda_i = \lambda_i' + a E(M, T_{\lambda_i'}, 1)$.  By altering
$U_i$ as necessary, we can write $M_i \cap U_i =
\graph_{T_{\lambda_i}}(u_i)$, and it's straightforward to check that
$E(M, T_{\lambda_i'}, 1)^{-1} u_i \to v - a r^\gamma\psi_1 = z$
smoothly on compact subsets of $\bC \cap B_1 \setminus \{ |x| = 0
\}$. 

Using property~\eqref{eq:triangle} we have $E(M, T_{\lambda_i}, 1) \leq
c(\bC) E(M, T_{\lambda_i'}, 1)$,
and by definition of $\lambda_i'$ we have $E(M, T_{\lambda_i'}, 1)
\leq E(M, T_{\lambda_i}, 1)$.  Therefore, after passing to a
subsequence, we can assume 
\[
\frac{E(M, T_{\lambda_i'}, 1)}{E(M, T_{\lambda_i}, 1)} \to b, \quad 1/c(\bC) \leq b \leq 1.
\]
In particular, we have smooth convergence $E(M, T_{\lambda_i}, 1)^{-1} u_i \to b z$.

By \eqref{eqn:decay-once-1} and Corollary \ref{cor:nc}, we have
\[
\limsup_i E(M, T_{\lambda_i}, 1)^{-1} E(M, T_{\lambda_i}, \rho) \leq c(\bC) \rho^{\gamma-1} \sup_{\bC \cap B_{2\rho}} ||x|^{-\gamma} b z| \leq c(\bC) \rho^{\gamma - 1 + 2\eps_1}
\]
for all $\rho \leq 1/4$.  Choose $\rho(\bC)$ sufficiently small so
that $c(\bC) \rho^{\eps_1} \leq 1$, we deduce  
\[
E(M, T_{\lambda_i}, \rho) \leq \rho^{\gamma - 1 + \eps_1} E(M, T_{\lambda_i}, 1)
\]
for all $i \gg 1$.  This is a contradiction, and finishes the proof of Proposition \ref{prop:decay-once}.
\end{proof}

\begin{prop}\label{prop:grow-away}
There are constants $\eps_2(\bC) > 0$, $c_2(\bC) > 1$ so that the following holds.  Let $M$ be a complete minimal hypersurface in $\R^{n+l+1}$, and suppose that $R^{-1} M \to [\bC]$ as varifolds as $R \to \infty$.  Then there is a $\lambda$ so that
\begin{equation}\label{eqn:grow-away-concl1}
E(M, T_\lambda, L R) \geq c_2(\bC)^{-1} L^{\gamma-1+\eps_2} E(M, T_\lambda, R)
\end{equation}
for all $L > 1$ and $R$ sufficiently large (depending only on $M$).  In particular, either $M = T_\lambda$, or there is a constant $C(M) > 0$ independent of $R$ so that
\begin{equation}\label{eqn:grow-away-concl2}
E(M, T_\lambda, R) \geq R^{\gamma-1+\eps_2}/C(M) \quad \forall R \geq C(M).
\end{equation}
\end{prop}

\begin{remark}\label{rem:change-lambda}
From \eqref{eq:triangle} and the scaling of $E$, if \eqref{eqn:grow-away-concl2} holds for some $\lambda$ then \eqref{eqn:grow-away-concl2} holds for any $\lambda'$, with a potentially larger $C(M, \lambda')$.
\end{remark}

\begin{proof}
Fix $\eps_2 = \eps = \min\{ \eps_0, \eps_1,1 \}/16$, $L_0 = \max \{ R_0(\bC, \eps), 1/\rho_1(\bC) \}$, let $\delta_0(\bC, \eps, R = L_0)$ be as in Lemma \ref{lem:3ann}, and let $\delta_1(\bC, \eta = \delta_0, \rho = 1/L_0)$ be as in Proposition \ref{prop:decay-once}.  By our hypothesis there is a radius $R_*$ so that for all $R \geq R_*$ we have
\[
E(M, \bC, R) < \min \{ \delta_0, \delta_1\}, \quad \theta_M(0, R) \leq (3/2) \theta_\bC(0).
\]

Apply Proposition \ref{prop:decay-once} to $R_*^{-1} M$ to obtain a $T_\lambda$, with $R_*^{1-\gamma} |\lambda| < \delta_0$, so that
\[
E(M, T_\lambda, R_*/L_0) L_0^{\gamma - 1 + \eps} \leq E(M, T_\lambda, R_*).
\]
By our choice of $\eps$, $L_0$, $R_*$, we can then apply Lemma \ref{lem:3ann} to $R_*^{-1} M$ to get
\[
E(M, T_\lambda, R_*) L_0^{\gamma - 1 + \eps} \leq E(M, T_\lambda, L_0 R_*).
\]
Now since $(R_* L_0)^{-1}T_\lambda = T_{ (R_* L_0)^{\gamma-1} \lambda}$, we can apply Lemma \ref{lem:3ann} again to $(L_0R_*)^{-1}M$ to get
\[
E(M, T_\lambda, L_0 R_*) L_0^{\gamma - 1+\eps} \leq E(M, T_\lambda, L^2_0 R_*).
\]
We can iterate to obtain
\[
E(M, T_\lambda, L^{k+l}_0 R_*) \geq L_0^{(\gamma - 1+ \eps)l} E(M, T_\lambda, L_0^kR_*), \quad \forall k, l \in \{0, 1, 2, \ldots \}.
\]
\eqref{eqn:grow-away-concl1} then follows with $c_2 = L_0^{\max\{ 2,
  \eps\}}= L_0^2$.  This completes the proof of Proposition \ref{prop:grow-away}.
\end{proof}

\section{One-sided decay and proof of main theorem}

\begin{prop}\label{prop:decay-away}
Let $M$ be a complete minimal hypersurface in $\R^{n+l+1}$ lying to one side of $\bC$, such that $\theta_M(\infty) < 2 \theta_\bC(0)$.  Then for any $\eps > 0$, we have
\begin{equation}\label{eqn:decay-away-concl1}
E(M, \bC, L R) \leq c_3(\bC, \eps) L^{\gamma - 1 + \eps} E(M, \bC, R)
\end{equation}
for all $L > 1$ and all $R$ sufficiently large (depending only on $M$).  In particular, there is a constant $C(M, \eps)$ independent of $R$ so that
\begin{equation}\label{eqn:decay-away-concl2}
E(M, \bC, R) \leq R^{\gamma - 1 + \eps} C(M, \eps) \quad \forall R \geq 1.
\end{equation}
\end{prop}

\begin{proof}
We first observe that by our hypotheses, the monotonicity formula, and Lemma \ref{lem:cone-liou} we must have $R^{-1} M \to [\bC]$ as varifolds as $R \to \infty$, and in particular we have $E(M, \bC, R) \to 0$ as $R \to \infty$.

If $M = \bC$ then the Proposition trivially holds, so assume $M \neq \bC$.  Fix $0 < \eps < \eps_0/16$ (there is no loss in assuming $\eps$ is as small as we like), and let $L_0 = R_0(\bC, \eps)$ as in Lemma \ref{lem:3ann}.  Suppose, towards a contradiction, there was a sequence $R_j \to \infty$ such that
\[
E(M, \bC, L_0 R_j) \geq L_0^{\gamma - 1 + \eps} E(M, \bC, R_j).
\]
Since $R_j^{-1} M \to \bC$, we can fix an $R_* = R_j$ sufficiently large and apply Lemma \ref{lem:3ann} successively to $R^{-1}_* M$, $(L_0 R_*)^{-1} M$, etc. to deduce
\begin{equation}\label{eqn:decay-away-1}
E(M, \bC, L_0^{k+l} R_*) \geq (L_0^k)^{\gamma - 1 + \eps} E(M, \bC, L_0^k R_*) \quad \forall k , l \geq 0
\end{equation}
By iterating \eqref{eqn:decay-away-1}, we deduce
\begin{equation}\label{eqn:decay-away-2}
E(M, \bC, L R) \geq c(\bC, \eps)^{-1} L^{\gamma - 1 + \eps} E(M, \bC, R)
\end{equation}
for all $L > 1$, and all $R \geq R_*$.

Choose $R_i \to \infty$ so that
\begin{equation}\label{eqn:decay-away-3}
2 E(M, \bC, R_i) \geq \sup_{R \geq R_i} E(M, \bC, R),
\end{equation}
and consider the rescaled surfaces $M_i = R_i^{-1} M$.  For $i \gg 1$, by Allard's theorem we can find an exhaustion $U_i$ of $\R^{n+l+1} \setminus \{|x| = 0\}$ so that $M_i = \graph_{\bC}(u_i)$.  From \eqref{eqn:decay-away-3}, we have
\[
2E(M_i, \bC, 1) \geq \sup_{R \geq 1} E(M, \bC, R) .
\]
Therefore by Corollary \ref{cor:nc}, after passing to a subsequence as
necessary, the rescaled graphs $E(M_i, \bC, 1)^{-1} u_i$ converge
smoothly on compact subsets of $\bC \setminus \{ |x| = 0 \}$ to a
Jacobi field $v$ satisfying
\begin{equation}\label{eqn:decay-away-4}
\sup_{\bC \cap B_R} R^{\gamma - 1} ||x|^{-\gamma} v| \leq c(\bC) \quad \forall R \geq 1.
\end{equation}
Moreover, since $M$ and hence $M_i$ all lie to one side of $\bC$, after flipping orientation as necessary we can assume $v \geq 0$.

Lemma \ref{lem:linear-liouville} implies $v(x = r\theta, y) = a r^\gamma \psi_1(\theta)$, where by \eqref{eqn:decay-away-4} $|a| \leq c(\bC)$.  From Corollary \ref{cor:nc}, for any $L > 1$ and any $i \gg 1$ we have
\[
\frac{E(M_i, \bC, L)}{E(M_i, \bC, 1)} \leq c(\bC) \sup_{\bC \cap B_{2L}} L^{\gamma - 1} ||x|^{-\gamma} v| \leq c(\bC) |a| L^{\gamma - 1} \leq c(\bC) L^{\gamma - 1},
\]
and hence
\begin{equation}\label{eqn:decay-away-5}
E(M, \bC, L R_i) \leq c(\bC) L^{\gamma - 1} E(M, \bC, R_i)
\end{equation}
for all $L > 1$ and $i$ sufficiently large, depending on $L$.

Combining \eqref{eqn:decay-away-2}, \eqref{eqn:decay-away-5} we get:
if $L > 1$, then for all $i$ large (depending on $L, M$) we have
\begin{equation}\label{eqn:decay-away-6}
c(\bC, \eps)^{-1}E(M, \bC, R_i) L^{\gamma - 1 + \eps} \leq E(M, \bC, L R_i) \leq c(\bC) L^{\gamma - 1} E(M, \bC, R_i) .
\end{equation}
Since $M \neq \bC$, for all $i \gg 1$ $E(M, \bC, R_i) \neq 0$, and so if we ensure $L(\bC, \eps)$ is sufficiently large \eqref{eqn:decay-away-6} will yield a contradiction.  Therefore, recalling our initial contradiction hypothesis, we must have
\begin{equation}\label{eqn:decay-away-7}
E(M, \bC, L_0 R) \leq L_0^{\gamma - 1 + \eps} E(M, \bC, R)
\end{equation}
for all $R \gg 1$.  Iterating \eqref{eqn:decay-away-7} gives \eqref{eqn:decay-away-concl1} and \eqref{eqn:decay-away-concl2}.
\end{proof}

\begin{proof}[Proof of Theorem \ref{thm:main}]
Assume that $M \neq T_\lambda$ for any $\lambda$.  As in the proof of Proposition \ref{prop:decay-away}, we have $R^{-1} M \to [\bC]$ as $R \to \infty$, and $E(M, \bC, R) \to 0$ as $R \to \infty$.  We can apply Proposition \ref{prop:grow-away} and Remark \ref{rem:change-lambda} to find a constant $C(M)$ so that
\begin{equation}\label{eqn:main-1}
E(M, \bC, R) \geq R^{\gamma-1+\eps_2} C^{-1} \quad \forall R \geq C.
\end{equation}
On the other hand, by Proposition \ref{prop:decay-away}, we can find another constant $C'(M)$ so that
\[
E(M, \bC, R) \leq R^{\gamma-1+\eps_2/2} C' \quad \forall R \geq 1,
\]
which contradicts \eqref{eqn:main-1} when $R \gg 1$.
\end{proof}

\end{document}